\documentclass[11pt, reqno, psamsfonts]{amsart}
\usepackage{graphicx}
\usepackage[backref,breaklinks]{hyperref}
\hypersetup{colorlinks, linkcolor={blue!60!black}, citecolor={blue!60!black}, urlcolor=cyan}
\usepackage[noabbrev,capitalize,nameinlink]{cleveref}
\usepackage{amsmath,amssymb,amsthm, bm}

\usepackage{placeins}
\usepackage{pgfplots}
\usepackage{csquotes} 
\pgfplotsset{compat=1.18}
\usepackage{xcolor}
\usepackage[ruled,vlined]{algorithm2e}
\usepackage{mathtools}
\usetikzlibrary{decorations.pathreplacing}

\theoremstyle{plain}
\newtheorem{theorem}{Theorem}[section]
\newtheorem*{theorem*}{Theorem}
\newtheorem{lemma}[theorem]{Lemma}
\newtheorem*{lemma*}{Lemma}
\newtheorem*{corollary*}{Corollary}

\newtheorem*{conjecture*}{Conjecture}

\newtheorem{corollary}[theorem]{Corollary}
\theoremstyle{definition}
\newtheorem{definition}[theorem]{Definition}

\theoremstyle{remark}
\newtheorem*{remark*}{Remark}

\newtheorem{claim}[theorem]{Claim}

\newtheorem*{question*}{Question}

\title{New Upper bounds on the Mondrian Art Problem}
\author{Thomas Garrison}
\thanks{Thomas Garrison: University of Florida,  \texttt{tgarrison2@ufl.edu}}

\author{Chris Seiler}
\thanks{Chris Seiler:  \texttt{cseiler@alumni.cmu.edu}}

\author{Aliaksei Semchankau}
\thanks{Aliaksei Semchankau: Carnegie Mellon University. \texttt{asemchan@andrew.cmu.edu}}

\begin{document}
\begin{abstract}
We present a new upper bound on the defect of the Mondrian Art Problem. The Mondrian Art Problem asks for a partition of an $n \times n$ square with rectangles of distinct dimensions such that the difference (defect) between the largest and smallest rectangle areas is minimized. We prove that for any $n \times n$ square, there exists a partition with defect $O(n^{5/6})$, improving upon the previously conjectured $O (n/\log n)$ upper bound. We also implement an algorithm that provides empirical evidence supporting our theoretical bound. 
\end{abstract}
\maketitle
\section{Introduction}\label{sec:intro}
Piet Mondrian's art has inspired many works across mathematics and beyond. The Mondrian art problem appears to have originated in recreational-mathematics circles around 2015--2016 \cite{Pegg2016Mondrian, Bassen2016FurtherII}, popularized by a Numberphile video featuring Gordon Hamilton. In the Mondrian Art Problem we want to partition an $n \times n$ square with distinct dimension rectangles such that the difference in area between the smallest and largest area rectangles is minimized (we count $m \times n$ and $n \times m$ as the same rectangle and not distinct ones). We define the \textit{defect} as the difference between the area of the smallest and largest rectangles in the partition. See \cref{fig:mondrian15} for an example of an optimal partition.
 
\begin{figure}[h]
    \centering
    \includegraphics[width=.8\linewidth]{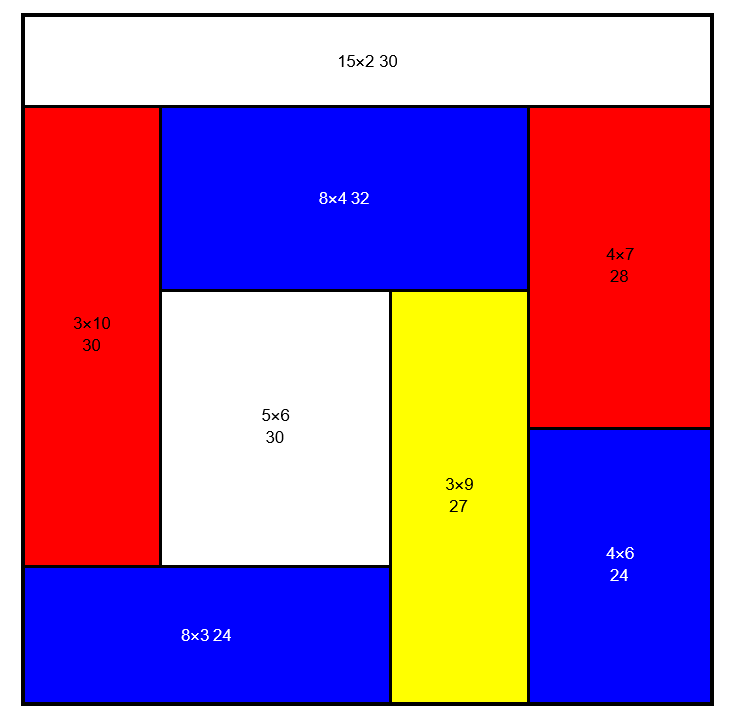}
    \caption{Mondrian partition of $15 \times 15$. This partition has a defect $32-24=8$ from Wolfram Demonstrations Project \enquote{Mondrian Art Problem} \cite{Pegg2016Mondrian}}
    \label{fig:mondrian15}
\end{figure}
Much computational work has been done on the Mondrian Art Problem, including calculating the minimum defect for squares of size $n=3-57$ \cite{OEIS:A276523, Bassen2016FurtherII}. Another big computation result in a different direction is a backtracking algorithm showing no square of sidelength less than or equal to 1000 can have a defect zero \cite{garciacolin2023perfectmondrianpartitionsquares}. Finally, a third approach to this problem uses number theory to show a lower bound on the density of values of $n$ with nonzero defects \cite{okuhn2018mondrianpuzzleconnectionnumber}. 
 
It was thought from computer experimentation that the best upper bound on the defects of an $n\times n$ square was $O(\frac{n}{\log n})$ \cite{OEIS:A276523}. We prove a polynomial improvement over the past upper bound.
 
\subsection{Organization}
This paper is organized as follows. In the second section, we prove an upper bound via our algorithm that yields a defect of $O(n^{\frac{5}{6}})$, with the implementation in the third section. 
 
\section{Proof of the upper bound}\label{sec:proof}
\begin{definition}
For $n \in \mathbb{N}$, a \textit{Mondrian partition} of a $n\times n$ square is a partition of the $n\times n$ square into disjoint, non-overlapping rectangles, with vertices that are integer coordinates such that each rectangle in the partition has distinct integer dimensions.
\end{definition}
In our proof of \cref{thm:main}, we plan to find a strip tiling of the $n \times n$ grid. A strip tiling consists of first partitioning the square into vertical strips, then tiling each strip with rectangles. The way we do this is, first, to show that the widths of all our strips sum to $n$ if we pick widths $w_i=\Theta(n^{1/2})$. Then we pick distinct heights of size $h_i=\Theta(n^{2/3})$ varying by $\Theta(n^{1/3})$. This gives us rectangles of area $a_i=\Theta(n^{7/6})$ with a defect of $d=\Theta(n^{5/6})$ as claimed. We invoke \cref{lem:subset} twice, once to show that the widths sum to $n$ and again to show that the heights will sum to $n$.
\begin{theorem}\label{thm:main}

There exists $n_0$ such that for $n\geq n_0$, there exists a Mondrian partition of the $n\times n$ square with defect  $d\leq 10 n^\frac{5}{6}$.

\end{theorem}
\begin{lemma}\label{lem:triangle}
   For any number $r$ in the set $\{0,1, \ldots, s(s + 1)/2\}$ there is a subset of $\{0,1, \ldots, s\}$ that sums to $r$.
\end{lemma}
This follows from the triangle number formula $\sum_{i=1}^s i=\frac{s(s+1)}{2}$.
\begin{lemma}\label{lem:subset}
Let $m, \ell$ be positive integers such that $m > \ell \geq \sqrt{2m}+1$. Fix $r \in \{0, 1, \ldots, m-1\}$ and let $s$ be the smallest positive integer such that $s(s + 1)/2 \geq r$. Then for any integer $k$ with $s \leq k \leq 2\ell - s$ there exists a subset of $\{m-\ell, \ldots, m+\ell\}$ summing to $mk+r$.
\end{lemma}
Before proving \cref{lem:subset}, let us prove a useful corollary.
\begin{corollary}\label{cor:range}
 For any number $N$ in the range $m(\sqrt{2m}+1) \leq N \leq 2m\ell - m(\sqrt{2m}+1)$ there is a subset $R_N \subseteq \{m - \ell, \ldots, m + \ell\}$ which sums up to $N$.    
\end{corollary}

\begin{proof}
    Write $N=mk+r$ with $r \in \{0,\ldots,m-1\}$. Let $s$ be minimal such that $s(s + 1)/2 \geq r$. $N\geq m(\sqrt{2m}+1)$. Dividing by $m$, we have $k+r/m\geq (\sqrt{2m}+1)$. Then we have $k\geq  (\sqrt{2m}+1)-\frac{r}{m} >  \sqrt{2m}\geq s$.
    For the upper bound, we have $N\leq 2m\ell - m(\sqrt{2m}+1)$. Dividing by $m$, we have $k\leq k+\frac{r}{m} \leq 2\ell-\sqrt{2m}-1\leq 2\ell-\sqrt{2m}\leq 2\ell-s$. Since we have both sides holding, we can use \cref{lem:subset} to conclude that we can find a subset summing to $N$. 
\begin{remark*}
The $+1$s in the range can be removed with a little more effort.
\end{remark*}
\end{proof}
The proof of \cref{lem:subset} is structured into two parts. In the first part, we want to get rid of the residue $r$ by adding up elements in the set $\{m + 1, \ldots, m + s\}$. In the second part of our proof we want the remaining tiles to add up to the number of multiples of $m$ we have left. This allows us to cover an entire row of the strip tiling.
\begin{proof}[Proof of \cref{lem:subset}]
  Part 1:\\
    Let us write $N = mk + r$, where $r \in \{0, \ldots, m - 1\}$.
  We have $s < \sqrt{2r}+1 \leq \sqrt{2m}+1$. We can see this because if $s$ is minimal, then $s(s - 1)/2 < r$. Then we have $(s-1)^2<s(s-1)<2r$ so $s<\sqrt{2r}+1$ and $r<m$ giving the desired inequalities.

    Let $A \subseteq \{1, \ldots, s\}$ be a subset that sums up to $r$ (assuming $r \neq 0$; otherwise we take the empty set). Setting $S_r := m + A = \{m + a \mid a \in A\}$ we have that the elements of this set sum to $m|A| + r$, where $k' := |A| \leq s$.
    Part 2:\\
    What remains is to find a set which sums up to $(mk + r) - (m|A| + r) = m(k - k')$, which we denote by $mk^{\prime\prime}$, $k^{\prime\prime} \in (0, 2\ell - s)$.
    Let $B := \{1, \ldots, \ell\} \setminus A = \{b_1 < b_2 < \ldots \}$. 
    It contains exactly $\ell - k' \geq \ell - s > 0$ numbers. The set $S_r$ contains no numbers of the form $m$ or $m \pm b$ with $b \in B$: its elements have the form $m + a$ with $a \in A$, and $B$ was chosen disjoint from $A$.
    We also notice that one can create all the multiples of $m$ up to $(2\ell - 2k' + 1)m$ as follows:
    \[
        \begin{split}
            m &= m, \\
            2m &= (m - b_1) + (m + b_1), \\
            3m &= (m - b_1) + m + (m + b_1), \\
            &\ldots \\
            (2\ell - 2k' + 1)m &= (m - b_1) + (m - b_2) + \ldots + m + \ldots + (m + b_2) + (m + b_1).
        \end{split}
    \]
Since $k' \leq s$ and $k < 2\ell - s$, we have $k+k'\leq 2\ell$, hence $k'' = k - k' \leq 2\ell - 2k' \leq 2\ell - 2k' + 1$.
    Therefore, there exists a set $S_m$ of numbers of the form $m$ and $m \pm b$, which sums up to $mk^{\prime\prime}$.
    
    Finally, setting $R_N := S_r \sqcup S_m$ we get
    \[
    \sum_{R_N} = \sum_{S_r} + \sum_{S_m} = (m|A| + r) + mk^{\prime\prime} = mk + r = N,
    \]
    which completes the proof.
\end{proof}
In our next proof, we prove \cref{thm:main}, showing that we can use this algorithm for tiling strips to partition the whole square and choose $O(n^\frac{7}{6})$ to be the size of our tiles and $O(n^\frac{5}{6})$ to be the size of our defect. 
\begin{figure}[h]
\centering
\begin{tikzpicture}[scale=0.5]
\draw[thick] (0,0) rectangle (12,12);
\draw[thick] (3,0) -- (3,12);
\draw[thick] (7,0) -- (7,12);
\draw[thick] (12,0) -- (12,12);
\draw[thick] (0,3.5) -- (3,3.5);
\draw[thick] (0,7.5) -- (3,7.5);
\draw[thick] (0,12) -- (3,12);
\draw[thick] (3,2.25) -- (7,2.25);
\draw[thick] (3,5) -- (7,5);
\draw[thick] (3,8.25) -- (7,8.25);
\draw[thick] (3,12) -- (7,12);
\draw[thick] (7,2.1) -- (12,2.1);
\draw[thick] (7,4.35) -- (12,4.35);
\draw[thick] (7,6.75) -- (12,6.75);
\draw[thick] (7,9.3) -- (12,9.3);
\draw[thick] (7,12) -- (12,12);
\node[below] at (1.5,-0.5) {$w_1$};
\node[below] at (5,-0.5) {$w_2$};
\node[below] at (9.5,-0.5) {$w_3$};
\node[right] at (12,1.2) {$h_{31}$};
\node[right] at (12,3.6) {$h_{32}$};
\node[right] at (12,6) {$h_{33}$};
\node[right] at (12,8.4) {$h_{34}$};
\node[right] at (12,10.8) {$h_{35}$};
\end{tikzpicture}
\caption{Example partition produced by the PartitionSquare algorithm. The width of each vertical strip is $\Theta(n^{1/2})$, and the heights within each strip are chosen so that all rectangle areas are approximately equal and have heights $\Theta(n^{2/3})$.}
\label{fig:partition-example}
\end{figure}
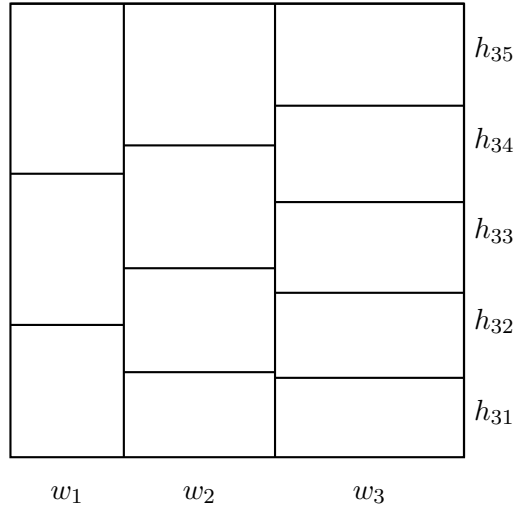
\begin{proof}[Proof of \cref{thm:main}]
In this proof we use four constants: $C$, the constant in front of the defect; $c$, the constant for the tile whose area we are approximating; $c_1, c_2$, the constants representing the smallest and largest strip widths.
Let $d := Cn^{5/6}$ be half of the max allowable defect ($C$ is permitted to be a large constant).

Let $T := cn^{7/6}$ be an area we will try to approximate ($c$ is a small constant). All areas will be in the $[T - d, T + d]$ range giving a defect of $2d$.

Now we partition the rectangle $n \times n$ into sub-rectangles $w_1 \times n$, $w_2 \times n$, $\ldots$ such that all $w_1, w_2, \ldots$ are distinct and come from the range $[c_1 n^{1/2}, \ldots, c_2 n^{1/2}]$. To make this partition cover the whole range, we just need 
\[
\sum_{c_1\sqrt{n}}^{c_2\sqrt{n}}w\geq n\ 
\Leftrightarrow\ 
\frac{c_2^2n}{2}-\frac{c_1^2n}{2}\geq n+o(n)\ 
\Leftrightarrow\ 
c_2^2 - c_1^2 \geq 2 + o(1).
\]

Each of the strips $w_i \times n$ we will partition into rectangles $w_i \times h_{i1}, w_i \times h_{i2}, \ldots$ in such a manner that  $w_i < h_{i1} < h_{i2} < \ldots$ holds, so as to ensure that all dimensions across different strips are distinct. Explicitly, take any two rectangles with dimensions $(w_i,h_{ij})$ and $(w_{i'},h_{i'j'})$; even if they have the same widths, their heights must be different so the rectangles have different dimensions. We also have $w_i \ll h_{ij}$ since $w_i=\Theta(n^{1/2})$ and $h_{ij}=\Theta(n^{2/3})$, so the heights and widths can never be equal, so we cannot have rectangles with height and width dimensions swapped. We also want all $w \times h_{ij}$ to have areas in $[T - d, T + d]$.

Given a strip $w_i \times n$, let us write $w = \alpha n^{1/2}$, $c_1 \leq \alpha \leq c_2$. 
Then we want all sides $h_{ij}$ to be in the interval $[T/w_i \pm d/w_i]$. In terms of the lemma above, we have $m = T / w_i \sim \frac{c}{\alpha}n^{2/3}, \ell = d / w_i \sim \frac{C}{\alpha}n^{1/3} $.

Notice, that as $T/w_i - d/w_i \gg n^{2/3} \gg n^{1/2}$, inequalities $w_i < h_{i1} < \ldots$ hold automatically.
\begin{claim}\label{clm:constants}
     We can choose constants $c,c_1,c_2,C$ such that all the numbers from $m\sqrt{2m}\leq n \leq 2m\ell - m\sqrt{2m}$ can be covered by the steps in our algorithm.   
\end{claim}

By \cref{cor:range}, we can pick heights for strips of width $w_i$ that sum to $n$ if $n$ lies in the range $[m(\sqrt{2m}+1), 2m\ell - m(\sqrt{2m}+1)]$. Since the $+1$ adds a term that is at most $O(n^{2/3})$, it is enough to show $n$ lies in $[m\sqrt{2m}, 2m\ell - m\sqrt{2m}]$ up to $o(n)$. Substituting in for $m =  \frac{c}{\alpha}n^{2/3}, \ell =  \frac{C}{\alpha}n^{1/3} $ this range becomes $ [\sqrt{2}c^{3/2}\alpha^{-3/2}n, (2Cc\alpha^{-2} - \sqrt{2}c^{3/2}\alpha^{-3/2})n]$. To have number $n = 1 \cdot n$ contained in this range we need inequalities
\[
\sqrt{2}\left(\frac{c}{c_1}\right)^{3/2} 
\leq 
1 
\leq 
2\frac{Cc}{c_2^2} - \sqrt{2}\left(\frac{c}{c_2}\right)^{3/2}
\]
and
\[
c_2^2 - c_1^2 \geq 2 + o(1)
\]
to hold.

We can do that. For instance, our inequalities are satisfied with constants $c = 1, c_1=2, c_2=2.5 ,  C = 5$.

This completes the proof.
\end{proof}
The number of tiles this algorithm uses is $\Theta (n^\frac{5}{6})$. 
\section{Implementation}\label{sec:impl}
The construction in \cref{sec:proof} lends itself to finding partitions of the $n \times n$ square efficiently. We provide an algorithm \textsc{PartitionSquare} which, given an integer $n$, produces a Mondrian partition of the $n \times n$ square with defect $O(n^{5/6})$.

Our main algorithm has a key component \textsc{RangeSubsetSum}, which given positive integers $m$ and a target value $N$, outputs a subset $S \subseteq [m-L, m+L] \cap \mathbb{N}$ such that $\sum_{x \in S}x = N$, for any positive integer value of $L$ satisfying $2m \sqrt m < N < 2m(L - \sqrt{2m})$ and $L < m$. In the Appendix, we provide \textsc{RangeSubsetSum} and the function it calls \textsc{BasicSubsetSum}. Here $c_0 = \frac{c_1+c_2}{2}$ denotes the median strip width constant, so that \textsc{RangeSubsetSum} returns distinct widths centered on $c_0\sqrt{n}$.

\begin{algorithm}
\SetAlgoLined
\KwIn{Parameter $n \in \mathbb{N}$}
\KwOut{A set of rectangles $R \subseteq (0, 1, 2, \dots, n)^4$, where each $(x_0, y_0, x_1, y_1) \in R$ denotes a rectangle with one corner at $(x_0, y_0)$ and the opposite corner at $(x_1, y_1)$}
\DontPrintSemicolon
\SetKwProg{Fn}{def}{}{}

\Fn{\normalfont \textsc{PartitionSquare}($n$)}{
  $R \gets \emptyset$\;
  $w_1, w_2, \dots, w_k \gets \textsc{RangeSubsetSum}(c_0 n^{1/2}, \text{target}=n)$\;
  \For{$i = 1$ \text{to} $k$}{
    $h_1, h_2, \dots, h_\ell \gets \textsc{RangeSubsetSum}(\frac{C n^{7/6}}{w_i}, \text{target}=n)$\;
    \For{$j = 1$ \text{to} $\ell$}{
       $x_0 \gets \sum_{r < i} w_r$\;
       $y_0 \gets \sum_{r < j} h_r$\;
       $x_1 \gets x_0 + w_i$\;
       $y_1 \gets y_0 + h_j$\;
       $R \gets R \cup \{ (x_0, y_0, x_1, y_1) \}$\;
    }
  }
  \Return{$R$}
}

\caption{\textsc{PartitionSquare}}
\end{algorithm}

Here, we have chosen the arbitrary constant $C$, as argued in the previous section, to be large enough to be valid for the input to \textsc{RangeSubsetSum}.

By \cref{thm:main},  every rectangle has area $Cn^{7/6}\pm O(n^{5/6})$, yielding defect $O(n^{5/6})$.

\begin{figure}[h!]
\centering
\begin{tikzpicture}
\begin{loglogaxis}[
    width=4in,
    height=3.2in,
    grid=major,
    xlabel={n},
    ylabel={Defects},
    legend pos=north west,
    tick label style={font=\small},
    label style={font=\small}
]

\addplot[only marks, mark=*, blue] coordinates {
(1000000000,495038016)
(2000000000,881782200)
(3000000000,1236662700)
(4000000000,1571044050)
(5000000000,1891385574)
(6000000000,2200712392)
(7000000000,2503172080)
(8000000000,2798502112)
(9000000000,3089255286)
(10000000000,3371595346)
(11000000000,3649631590)
(12000000000,3923492482)
(13000000000,4192822368)
(14000000000,4460980446)
(15000000000,4725535640)
(16000000000,4984845762)
(17000000000,5245558722)
(18000000000,5499171371)
(19000000000,5753314279)
(20000000000,6003318132)
};

\addplot[red, dashed, domain=1e9:2e10, variable=n] {n^.833333333333*15.63866567};
\legend{Data points, $15.63866567*n^{\frac{5}{6}}$ reference}

\end{loglogaxis}
\end{tikzpicture}
\caption{We use our code to create a plot of defects of squares of size $1*10^9-2*10^{10}$ in increments of $10^9$.}

\label{fig:defectplot}

\end{figure}
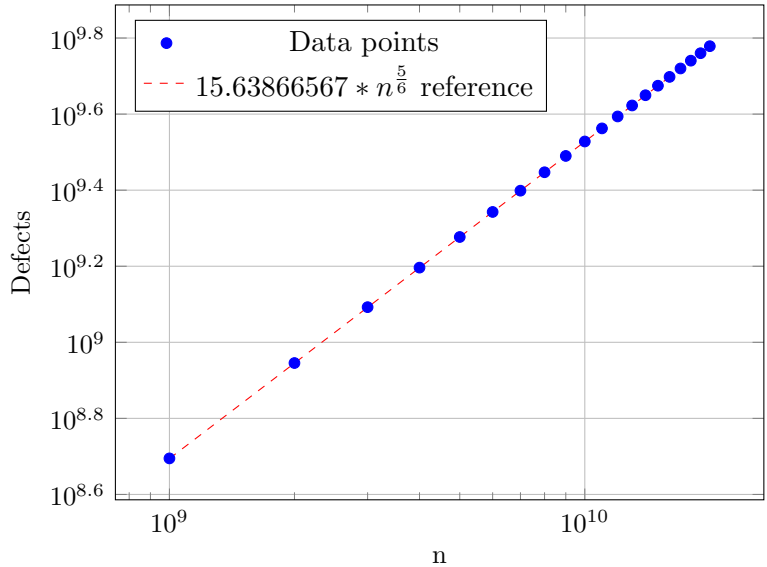
\FloatBarrier
\section{Concluding remarks}
We analyze strip tilings and apply them to lower the upper bound on the defect of the Mondrian Art Problem. We proved an upper bound of $n^\frac{5}{6}$ for large $n$. Then we show an algorithm that achieves this upper bound.

While trying to find optimal Mondrian partitions we have observed that often the limiting factor is the largest shared face (either two rectangles next to each other or two rectangles that share one of the same dimensions). Face-sharing inherently limits the size of tiles you can use by forcing shared dimensions. With less or no face-sharing, we can use more distinct factorizations of tilings than with face-sharing. Finally, for a Mondrian partition to have defect zero, it could not have any face-sharing because two tiles that share one face and have the same area must have the same dimension, which is disallowed. This leads us to two questions. \begin{enumerate}
    \item Can you get a better upper bound if face-sharing is disallowed?
    \item How can one create families of these partitions?
\end{enumerate}

Another open question we have is:
Are there more optimal algorithms that give a better upper bound on the defect?
\appendix
\section{}

\begin{algorithm}
\SetAlgoLined
\KwIn{Parameter $L \in \mathbb{N}$; Target value $N$}
\KwOut{A subset $S \subseteq [L]$ such that $\sum_{x \in S} x = N$}
\DontPrintSemicolon
\textbf{Precondition:} $1 \leq N \leq \frac{L(L+1)}{2}$\;
\textbf{Runtime:} $O(L)$\;
$S \gets \emptyset$\;
\For{$i \gets L$ \textbf{downto} $1$}{
    \If{$i + \sum S \leq N$}{
        $S \gets S \cup \{i\}$\;
    }
}
\Return{$S$}\;
\caption{\textsc{BasicSubsetSum}}
\end{algorithm}

The correctness of \textsc{BasicSubsetSum} is straightforward.

\begin{algorithm}
\SetAlgoLined
\KwIn{Parameters $m, L \in \mathbb{N}$; Target value $N$}
\KwOut{A subset $S \subseteq [m - L, m + L] \cap \mathbb{N}$ such that $\sum_{x \in S} x = N$}
\DontPrintSemicolon
\SetKwProg{Fn}{def}{}{}

\textbf{Precondition:} $2m\sqrt{m} < N < 2m(L - \sqrt{2m})$ and $L < m$\;

\Fn{\normalfont \textsc{RangeSubsetSum}$(m, L, N)$}{
    $k, r \leftarrow \text{divmod}(N,m)$ \tcp*{$N = km + r$, $0 \leq r < m$}
    $S^+ \leftarrow \textsc{BasicSubsetSum}(\sqrt{2m}, \text{target}=r)$\;
    $S^- \leftarrow \emptyset$\;
    $w \leftarrow \left\lceil\sqrt{2m}\right\rceil + 1$\;
    \If{$k - |S^+|$ \normalfont{is odd}}{
        $S^- \leftarrow S^- \cup \{m\}$\;
    }
    \While{$|S^-| < k - |S^+|$}{
        $S^- \leftarrow S^- \cup \{m-w, m+w\}$\;
        $w \leftarrow w + 1$\;
    }
    \Return{$\{m+x: x \in S^+\} \cup S^-$}\;
}
\caption{\textsc{RangeSubsetSum}}
\end{algorithm}

To show correctness of \textsc{RangeSubsetSum}, note that $\sum S^- = m \cdot |S^-|$, an invariant we maintain at every iteration. Further,\\
$|S^-| = k - |S^+|$, so $\sum S^- = m(k - |S^+|)$.

Now,
\begin{align*}
\sum S &= \sum_{x \in S^+} (m+x) + m(k - |S^+|)\\
&= m|S^+| + \sum S^+ + m(k - |S^+|)\\
&= mk + \sum S^+ = mk + r = N.
\end{align*}
\newpage
\bibliographystyle{amsplain}

\bibliography{mondrian1}

\end{document}